\documentclass[11pt]{amsart}
\usepackage{amsmath, amsfonts, eucal, upref}

\newcommand{\xxi}{\dot{\sigma}}

\newtheorem{thm}{Theorem}[section]
\newtheorem{lmm}[thm]{Lemma}
\newtheorem{cor}[thm]{Corollary}

\newcommand{\avg}[1]{\bigl\langle #1 \bigr\rangle}

\newcommand{\bigavg}[1]{\biggl\langle #1 \biggr\rangle}

\newcommand{\ee}{\mathbb{E}}

\newcommand{\ra}{\rightarrow}
\newcommand{\rr}{\mathbb{R}}
\newcommand{\smallavg}[1]{\langle #1 \rangle}

\newcommand{\xp}{X^\prime}

\newcommand{\xx}{\mathcal{X}}

\newcommand{\fpar}[2]{\frac{\partial #1}{\partial #2}}

\newcommand{\mpar}[3]{\frac{\partial^2 #1}{\partial #2 \partial #3}}

\begin{document}
\title{Spin glasses and Stein's method}
\author{Sourav Chatterjee}
\thanks{The author's research was partially supported by NSF grant DMS-0707054 and a Sloan Research Fellowship}
\address{\newline367 Evans Hall \#3860\newline
Department of Statistics\newline
University of California at Berkeley\newline
Berkeley, CA 94720-3860\newline
{\it E-mail: \tt sourav@stat.berkeley.edu}\newline 
{\it URL: \tt http://www.stat.berkeley.edu/$\sim$sourav}}

\begin{abstract}
We introduce some applications of Stein's method in the high temperature analysis of spin glasses. Stein's method allows the direct analysis of the Gibbs measure without having to create a cavity. Another advantage is that it gives limit theorems with total variation error bounds, although the bounds can be suboptimal. A surprising byproduct of our analysis is a relatively transparent explanation of the Thouless-Anderson-Palmer system of equations. Along the way, we develop Stein's method for mixtures of two Gaussian densities. 
\end{abstract}

\maketitle

\section{Introduction and results}\label{intro}
\subsection{The Sherrington-Kirkpatrick model}
Let $N$ be a positive integer and let $\Sigma_N = \{-1,1\}^N$. A typical element of $\sigma = (\sigma_1,\ldots,\sigma_N)\in \Sigma_N$ is called a `configuration' of spins. Let $g = (g_{ij})_{1\le i<j\le N}$ be a collection of independent standard Gaussian random variables, called the `disorder' in our context. Given a realization of $g$, fix any $\beta > 0$ and $h\in \rr$ and define a probability distribution $G_N$ on $\Sigma_N$  as 
\[
G_N(\sigma) = Z_N^{-1}\exp\biggl(\frac{\beta}{\sqrt{N}} \sum_{1\le i<j\le N} g_{ij} \sigma_i \sigma_j + h\sum_{1\le i\le N} \sigma_i\biggr).
\]
Here $Z_N = Z_N(\beta,h)$ is the normalizing constant (partition function), and $G_N$ is the `Gibbs measure'. What we have just defined is the well-known Sherrington-Kirkpatrick (SK) model of spin glasses \cite{sk75}. The parameter $\beta$ is called the `inverse temperature' of the model, and $h$ is called the `external magnetic field', following the conventions of Talagrand \cite{talagrand03}. Physicists would replace the $h$ with $\beta h$, but the definitions are mathematically equivalent.

Configurations chosen independently from the Gibbs measure (i.e., given the disorder) are denoted by $\sigma^1, \sigma^2$, etc. These are called `replicas' in physics. 
If $\sigma^1,\ldots,\sigma^k$ are replicas and $f$ is a function on $\Sigma_N^k$, then as usual we define
\[
\avg{f(\sigma^1,\ldots,\sigma^k)} := \sum_{\sigma^1,\ldots,\sigma^k} f(\sigma^1,\ldots,\sigma^k) G_N(\sigma^1)\cdots G_N(\sigma^k).
\]
In the terminology of disordered systems, $\smallavg{f(\sigma)}$ is known as the quenched average of $f(\sigma)$.%, and $\ee \smallavg{f}$ is called the annealed expectation. 
%We will freely use this terminology throughout the discussion.

The `overlap' between a pair of replicas $\sigma^1$ and $\sigma^2$, chosen independently from the Gibbs measure, is defined as
\begin{equation}\label{r12}
R_{12} := \frac{1}{N}\sum_{i=1}^N \sigma_i^1 \sigma_i^2.
\end{equation}
The `high temperature phase' of the SK model corresponds to the set of $(\beta,h)$ for which  there is a number $q = q(\beta, h)< 1$ such that the overlap $R_{12}$ is approximately equal to $q$ with high probability under the Gibbs measure. This can be made precise in various ways, and the form that will be most suitable for us in this article is:
\begin{equation}\label{high}
\ee\avg{(R_{12}-q)^4} \le \frac{C(\beta,h)}{N^2}, 
\end{equation}
where $C(\beta, h)$ is a constant that depends only on $\beta$ and $h$.
(At this point, let us declare that throughout this paper, statements like  ``$T\le C(\beta,h)$'' stands for ``the term $T$ can be bounded by a constant that depends only on $\beta$ and~$h$''.)
It is known (see e.g.\ \cite{talagrand03}, p.\ 72) that the constant $q$ must satisfy 
\begin{equation}\label{qdef}
q = \ee \tanh^2(\beta z\sqrt{q}  + h),
\end{equation}
where $z$ is a standard Gaussian random variable.

It is not very difficult to show that a consequence of the concentration of the overlap is that small collections of spins become approximately independent under the Gibbs measure (see \cite{talagrand03}, Theorem 2.4.10). However, they are not identically distributed unless $h = 0$. One important objective of the theory of spin glasses is to find ways to compute the marginal distributions of the spins. A way to do this is via the Thouless-Anderson-Palmer (TAP) equations, which we study later in this article.

The high temperature phase of the SK model under zero external field was studied rigorously by  Aizenman, Lebowitz, and Ruelle~\cite{alr87}. A more systematic and powerful approach via stochastic calculus was developed by Comets and Neveu \cite{cometsneveu95} and extended by Tindel \cite{tindel05}. The high temperature phase for $h \ne 0$ was rigorously investigated by Fr\"ohlich and Zegarli\'nski \cite{frohlichzegarlinski87} and more extensively by Shcherbina \cite{shcherbina97} and Talagrand~\cite{talagrand98}. An extremely thorough  rigorous treatment of the high temperature phase with many new results appeared in Chapter 2 of Talagrand's book \cite{talagrand03}. An important result, shown in~\cite{talagrand03}, Theorem 2.5.1, is that there exists a constant $\beta_0 >0$ such that \eqref{high} holds whenever $\beta \le \beta_0$. In this manuscript, this is only result we borrow from the existing theory of spin glasses. We did not attempt to prove this via Stein's method. 

As of now, even the low temperature phase is somewhat mathematically tractable, following the deep contributions of Guerra \cite{guerra03}, Guerra and Toninelli \cite{guerra02}, Talagrand \cite{talagrand06}, and Panchenko \cite{panchenko}. The recent paper of Comets, Guerra, and Toninelli \cite{cometsetal05} connecting the SK model and its lattice counterpart is also of interest.
For a review of the extensive but mostly unrigorous developments in the theoretical physics literature, let us refer to the classic text of M\'ezard, Parisi, and Virasoro \cite{mezardetal87}.

\subsection{The TAP equations}
Since the spins can take only two values, the quenched distribution of the spin at site $i$ is completely described by the quantity $\smallavg{\sigma_i}$. One of the main approaches (as outlined in \cite{mezardetal87}) to understanding the high temperature phase of the SK model is to understand the quantities $\smallavg{\sigma_1},\ldots,\smallavg{\sigma_N}$ via the Thouless-Anderson-Palmer system of equations: 
\begin{equation}\label{tapeq}
\avg{\sigma_i} \approx \tanh\biggl(\frac{\beta}{\sqrt{N}}\sum_{j\ne i} g_{ij}\avg{\sigma_j} + h - \beta^2 (1-q)\avg{\sigma_i}\biggr), \ \ i =1, \ldots, N.
\end{equation}
Here $\approx$ means, vaguely, `approximately equal with high probability'. Physicists usually write exact equalities in such cases.

This self-consistent system of equations has a unique solution with high probability if $\beta$ is small. It was physically argued by Thouless, Anderson, and Palmer \cite{tap77} that the quantities $\smallavg{\sigma_1},\ldots,\smallavg{\sigma_N}$ must satisfy these equations `in the large $N$ limit' at any temperature and external field. The first rigorous proof of the validity of the TAP equations in the high temperature phase (where \eqref{high} holds) appeared twenty-six years after the publication of the physics paper, in Talagrand's book~(\cite{talagrand03}, Theorem 2.4.20). However, Talagrand's theorem in \cite{talagrand03} does not show that {\it all} $N$ equations hold simultaneously with high probability. This has been proved more recently (Talagrand, private communication), and is going to appear in the forthcoming edition of \cite{talagrand03}.

Talagrand's proof is based on a remarkable rigorous formulation of the cavity method, which involves studying the system after `removing the last spin'. This procedure is known as `creating a cavity'. Now, if one wants to study the Gibbs measure directly, without having to resort to the essentially inductive process of creating a cavity, is there a way to proceed? This is a key focus in this paper. Let us begin by outlining the approach for the case of the TAP equations and understanding how they arise.

\subsection{The Onsager correction term}
The first step is to observe that the conditional expectation of $\sigma_i$ given $(\sigma_j)_{j\ne i}$ under the Gibbs measure is simply $\tanh(\beta \ell_i + h)$, where $\ell_i$ is the {\it local field} at site $i$, defined as
\begin{equation}\label{li}
\ell_i = \ell_i(g,\sigma) := \frac{1}{\sqrt{N}} \sum_{j\ne i} g_{ij} \sigma_j.
\end{equation}
The proof of this is quite trivial, following directly from the form of the Gibbs measure. It follows that
\begin{equation}\label{callen}
\avg{\sigma_i} = \avg{\tanh(\beta \ell_i + h)}, \ \ i=1,\ldots,N.
\end{equation}
%These are the so-called {\it Callen equations}. 
Thus, if we could understand the distribution of the local fields under the Gibbs measure, the problem of computing $\langle \sigma_i\rangle$ would be solved. This motivates the study of the limiting behavior of the local fields.

Incidentally, the na\"ive mean field heuristic would dictate that the `average can be moved inside the $\tanh$', and
\begin{align*}
\avg{\sigma_i} = \avg{\tanh(\beta \ell_i + h)} &\stackrel{?}{\approx} \tanh\bigl(\beta \avg{\ell_i} + h\bigr)\\
&=\tanh\biggl(\frac{\beta}{\sqrt{N}}\sum_{j\ne i} g_{ij}\avg{\sigma_j} + h\biggr).
\end{align*}
However, the na\"ive heuristic does not work for the SK model. This is not surprising since the local fields are unlikely to be concentrated. The famous discovery of Thouless, Anderson, and Palmer \cite{tap77} is that the average can still be moved inside the $\tanh$, but only after adding what has come to be known as the {\it Onsager correction term}, $-\beta^2(1-q)\avg{\sigma_i}$. As stated before, this gives the TAP equations
\begin{equation*}
\avg{\sigma_i} \approx \tanh\biggl(\frac{\beta}{\sqrt{N}}\sum_{j\ne i} g_{ij}\avg{\sigma_j} + h - \beta^2 (1-q)\avg{\sigma_i}\biggr), \ \ i =1, \ldots, N.
\end{equation*}
In view of the approximate independence of the spins under the Gibbs measure at high temperature, it seems natural to surmise that the local fields would be approximately Gaussian. Surprisingly, this is not the case. Rather, the explanation for the Onsager correction is hidden in a property of convex combinations of pairs of Gaussian distributions.

\subsection{Onsager correction and mixture Gaussians}
For any $\mu\in \rr$, $\sigma > 0$, let $N(\mu, \sigma^2)$ denote the Gaussian distribution with mean $\mu$ and variance $\sigma^2$, with density function
\begin{equation}\label{normdens}
\phi_{\mu,\sigma^2}(x) = \frac{1}{\sqrt{2\pi} \sigma} e^{-(x-\mu)^2/2\sigma^2}.
\end{equation}
The mixture (i.e.\ convex combination) of two Gaussian densities has a curious connection with the $\tanh$ function. Suppose $X$ is a random variable following the mixture density
$p\phi_{\mu_1,\sigma^2} + (1-p)\phi_{\mu_2,\sigma^2}$. 
Suppose $\mu_1 > \mu_2$, and let
\begin{align}\label{ab}
a = \frac{\mu_1 - \mu_2}{2\sigma^2}, \ b = \frac{1}{2}\log \frac{p}{1-p} - \frac{\mu_1^2 - \mu_2^2}{4\sigma^2}.
\end{align}
Then a simple computation gives
\begin{equation}\label{tanh}
\ee \tanh(aX + b) = \tanh(a\ee(X) + b - (2p -1)a^2\sigma^2).
\end{equation}
That is, the `expectation can be moved inside the $\tanh$', after incurring the quadratic `correction term' $-(2p-1) a^2 \sigma^2$. (The proof of this identity is sketched in the Subsection \ref{steinmix}. With other values of $a$ and $b$, the expectation can still be moved inside, but the correction term will no longer have a simple form.) The similarity with the Onsager correction is more than superficial: in fact, it turns out that the distribution of the local fields under the Gibbs measure can be approximated by a mixture of two Gaussian densities, and the correction term in the above equation indeed corresponds to the Onsager correction term in the TAP equations.

\subsection{Limit law for the local fields}
The precise result about the limiting distribution of the local fields can be described in the following way. For each $1\le i\le N$, let
\begin{align}
r_i &= r_i(g) := \frac{1}{\sqrt{N}} \sum_{j\ne i} g_{ij} \avg{\sigma_j} - \beta(1-q)\avg{\sigma_i}, \label{ri} \\
p_i &:=\frac{e^{\beta r_i + h}}{e^{\beta r_i + h} + e^{-\beta r_i - h}}, 
\end{align}
and let $\nu_i$ be the random probability measure on $\rr$ with the mixture Gaussian density
\[
p_i \phi_{r_i + \beta (1-q), 1-q} + (1-p_i)\phi_{r_i - \beta (1-q), 1-q}.
\]
Then the distribution of the local field $\ell_i$ (defined in \eqref{li}) under the Gibbs measure is close to $\nu_i$, in the sense that the difference between the two (random) measures converges in probability to the zero measure. A more quantitative result is as follows.
\begin{thm}\label{mainthm}
Suppose $\beta$ and $h$ are such that \eqref{high} is satisfied for some $q< 1$. Let $\nu_1,\ldots,\nu_N$ be defined as above. Then for any bounded measurable $u:\rr \ra \rr$ and any $1\le i\le N$, we have
\[
\ee\biggl(\avg{u(\ell_i)} - \int_{\rr} u(x) \nu_i(dx)\biggr)^2 \le \frac{C(\beta,h) \|u\|_\infty^2}{\sqrt{N}},
\]
where $C(\beta,h)$ is a constant depending only on $\beta$ and $h$.
\end{thm}
\noindent 
Taking $u(x) = \tanh(\beta x + h)$ and using the connection \eqref{tanh} between mixture Gaussian distributions and the Onsager correction, we can now readily prove the TAP equations.
\begin{cor}\label{tap}
If \eqref{high} is satisfied, then for each $1\le i\le N$ we have
\[
\ee\biggl(\avg{\sigma_i} - \tanh\biggl(\frac{\beta}{\sqrt{N}}\sum_{j\ne i} g_{ij}\avg{\sigma_j} + h - \beta^2 (1-q)\avg{\sigma_i}\biggr)\biggr)^2 \le \frac{C(\beta,h)}{\sqrt{N}}.
\]
\end{cor}
\begin{proof}
By \eqref{callen}, we know that $\avg{\sigma_i} = \avg{\tanh(\beta \ell_i + h)}$. Now, for the mixture Gaussian density $\nu_i$, a simple computation shows that $a = \beta$ and $b = h$, where $a$ and $b$ are defined as in \eqref{ab}. Taking $u(x) = \tanh(\beta x + h)$ in Theorem~\ref{mainthm} and using the property \eqref{tanh} of mixture Gaussian densities, it is not difficult to verify that we get the stated result.
\end{proof}
\noindent Note that Talagrand's version of Corollary \ref{tap} (\cite{talagrand03}, Theorem 2.4.20) has an error bound of order $1/N$, and so our result is suboptimal. However, we are not entirely certain whether Theorem \ref{mainthm} itself is suboptimal (although it probably is), because improving the $1/\sqrt{N}$ bound in the proof seems to require some kind of smoothness for the function $u$, which we are not assuming.

\subsection{An explanation of the mixture Gaussianity}
%An intuitive explanation for the mixture gaussian nature of $\ell_i$ is not difficult to produce. 
Since we know the conditional distribution of $\sigma_i$ given $\ell_i$, %(because the conditional law of $\sigma_i$ given $(\sigma_j)_{j\ne i}$ depends only on $\ell_i$), 
and we know the marginal laws of $\ell_i$ and $\sigma_i$ via Theorem \ref{mainthm} and Corollary \ref{tap}, it is possible to compute the conditional law of $\ell_i$ given $\sigma_i$ by Bayes' rule. It turns out that given $\sigma_i = 1$, the law of $\ell_i$  is approximately Gaussian with mean $r_i + \beta(1-q)$ and variance $1-q$, and given $\sigma_i = -1$, the law of $\ell_i$ is approximately Gaussian with mean $r_i - \beta(1-q)$ and variance $1-q$.  Thus, the marginal distribution of $\ell_i$ under the Gibbs measure is approximately a convex combination of these two distributions.

\subsection{Stein's method}
%Let us now proceed to develop the tools for proving Theorem \ref{mainthm} and other results by directly analyzing the Gibbs measure using Stein's method (i.e.\ instead of creating a cavity, as per our stated goal). Briefly, Stein's method can be described as follows.
We prove Theorem \ref{mainthm} using our version of the classical probabilistic tool developed by C. Stein \cite{stein72, stein86}. Incidentally, it is also possible to prove it using standard techniques from the cavity method as developed by Talagrand. However, as we shall see below, one advantage of Stein's method, besides the total variation error bounds, is that it allows us to `discover' the result before proving it. Let us give a brief primer on Stein's method below.

Suppose we want to show that a random variable $X$  has approximately the same distribution as some other random variable $Z$. The basic idea behind Stein's method of distributional approximation~\cite{stein72, stein86} is as follows.
\begin{enumerate}
\item[1.] Identify a ``Stein characterizing operator'' $T$ for $Z$, which has the defining property that for any function $f$ belonging to a fixed large class of functions, $\ee Tf (Z) = 0$. For instance, if $Z$ is a standard Gaussian random variable, then $Tf(x) := f^\prime(x)-xf(x)$ is a characterizing operator, acting on all locally absolutely continuous~$f$.  
\item[2.] Take a function $u$ and find $f$ such that $Tf(x) = u(x) - \ee u(Z)$. Relate the smoothness properties of $f$ to those of $u$. 
\item[3.] By the definition of $f$ it follows that $|\ee u(X) - \ee u(Z)| = |\ee(Tf(X))|$. Compute a bound on $|\ee (Tf(X))|$ by whatever means possible.
\end{enumerate}
The procedure for normal approximation can be simply described as follows: if we want to show that a random variable $X$ is approximately standard Gaussian, Stein's method demands that we show $\ee (f'(X) - Xf(X)) \approx 0$ for every $f$ belonging to a large class of functions.

Although the raw version of the method as stated above may seem like a trivial reduction, the replacement of $u(x)-\ee u(Z)$ by $Tf(x)$ often gives a high degree of maneuverability in practice. While steps 1 and 2 have to be carried out exactly once for every distribution of $Z$, the execution of step~3 depends heavily on the problem at hand. A number of techniques for carrying out this step are available in the literature, e.g.\ exchangeable pairs~\cite{stein86}, diffusion generators~\cite{barbour90}, dependency graphs~\cite{baldirinottstein89, agg90}, size bias couplings~\cite{goldsteinrinott97}, zero bias couplings~\cite{goldsteinreinert97}, couplings for Poisson approximation~\cite{chen75, bhj92}, specialized procedures like~\cite{rinottrotar96, rinottrotar97, fulman04}, and some recent advances~\cite{chatterjee06, chatterjee07, chatterjeefulman06, chatterjeemeckes06}. Incidentally, Stein's method was applied to solve a problem in the interface of statistics and spin glasses in~\cite{chatterjee06a}.

\subsection{Stein's method for mixture Gaussians}\label{steinmix}
For any $a, b, \mu\in \rr$ and $\sigma^2 > 0$, let $M(a,b,\mu, \sigma^2)$ be the probability distribution on $\rr$ with density function
\begin{equation}\label{convenient}
\psi_{a,b,\mu,\sigma^2}(x) = Z_{a,b,\mu,\sigma^2}^{-1}\cosh(ax + b) e^{-(x-\mu)^2/2\sigma^2}.
\end{equation}
where the normalizing constant is given by
\[
Z_{a,b,\mu,\sigma^2} = \sqrt{2\pi} \sigma \cosh(a\mu + b) e^{\frac{1}{2}a^2\sigma^2}.
\]
A simple verification shows that $M(a,b,\mu, \sigma^2)$ is in fact a mixture of two Gaussian distributions:
\begin{equation}\label{mix}
\psi_{a,b,\mu,\sigma^2}(x) = p\phi_{\mu + a\sigma^2, \sigma^2}(x) + (1-p) \phi_{\mu - a\sigma^2, \sigma^2}(x),
\end{equation}
where $\phi$ stands for the Gaussian density function \eqref{normdens} and 
\[
p = \frac{e^{a\mu+b}}{e^{a\mu+b} + e^{-a\mu-b}}.
\]
Thus, the distributions representable as $M(a,b,\mu,\sigma^2)$ are exactly the distributions arising as mixtures of two Gaussian densities. Note that in particular, $M(0,0,\mu,\sigma^2)$ is just the Gaussian distribution with mean $\mu$ and variance $\sigma^2$. An interesting fact about this class of distributions, required for the proof of Corollary \ref{tap}, is that
\begin{equation}\label{tanh2}
\int_{\rr}\tanh(ax + b) \psi_{a,b,\mu,\sigma^2}(x) dx = \tanh(a\mu + b).
\end{equation}
The computation can be easily done using the convenient representation~\eqref{convenient}. Note that this is exactly the relation \eqref{tanh}.
Again, using \eqref{convenient} it is not difficult to verify that the operator
\[
Tf(x) = f'(x) - \biggl(\frac{x-\mu}{\sigma^2} - a\tanh(ax+b)\biggr) f(x)
\]
is a Stein characterizing operator for $M(a,b,\mu,\sigma^2)$. Roughly, this means that to show that a random variable $W$ approximately follows the distribution $M(a,b,\mu,\sigma^2)$, we have to show that for all $f$, 
\[
\ee\biggl(f'(W) - \biggl(\frac{W-\mu}{\sigma^2} - a\tanh(aW+b)\biggr) f(W)\biggr) \approx 0.
\]
To develop Stein's method for this class of distributions, we have to solve
\begin{equation}\label{tf}
Tf(x) = u(x) - \int_{\rr} u(t)\psi_{a,b,\mu,\sigma^2} dt
\end{equation}
for arbitrary $u:\rr\ra\rr$, and relate bounds on $f$ and its derivatives to properties of $u$.

Now note that by \eqref{mix}, the measure $\nu_i$ in Theorem \ref{mainthm} can be alternatively written as $M(\beta,h,r_i,1-q)$. The randomness of $r_i$ adds an extra level of complexity: We also have to analyze the dependence of the function $f$ in~\eqref{tf} on the parameter $\mu$. So we should start with  $f(x,\mu)$ rather than $f(x)$.
The bounds required for Stein's method are summarized in the following lemma. 
\begin{lmm}\label{steinlmm}
Fix $a, b\in \rr$ and $\sigma^2 > 0$, and a bounded  measurable function $u: \rr \ra \rr$. Then there exists an absolutely continuous function $f:\rr^2 \ra \rr$ such that for all $x,\mu\in \rr$,
\begin{align*}
\fpar{f}{x}(x,\mu) - \biggl(&\frac{x-\mu}{\sigma^2} - a\tanh(ax+b)\biggr) f(x, \mu) \\
&= u(x) - \int_{\rr} u(t) \psi_{a,b,\mu,\sigma^2}(t) dt.
\end{align*}
Moreover, we can find a solution $f$ such that for some  constant $C(a,\sigma)$ depending only on $a$ and $\sigma$ we have that $|f|$, $\bigl|\frac{\partial f}{\partial x}\bigr|$, and $\bigl|\frac{\partial f}{\partial \mu}\bigr|$ are all uniformly bounded by $C(a,\sigma) \|u\|_\infty$. 
\end{lmm}
\noindent Note that the case $a = b = 0$ covers the case of the pure Gaussian distributions. The proof of this lemma, which is quite elementary but tedious, is relegated to the end of the manuscript.

\subsection{How to apply Stein's method}
To see how Stein's method can be used in the SK model, let us sketch a very simple example: the unconditional (i.e.\ average over the disorder) distribution of the local field at site $1$ when $\beta < 1$ and $h = 0$. This is a special case of Theorem \ref{mainthm}. Since $R_{1,2}$ concentrates around zero in this regime (see e.g.\ \cite{alr87}, or \cite{talagrand03}, Chapter 2), we have $q = q(\beta,0) = 0$. Also, by symmetry, $\smallavg{\sigma_i}\equiv 0$ for each $i$. Therefore the measure $\nu_i$ is actually a nonrandom probability measure, namely, the mixture Gaussian density $\frac{1}{2}\phi_{\beta,1} + \frac{1}{2}\phi_{-\beta,1}$. Clearly, the nonrandomness of the limiting distribution hugely simplifies our goal. Let us now see how we can prove via Stein's method that this is the limiting distribution of the local fields.  

Recall that
$\ell_1 = \frac{1}{\sqrt{N}} \sum_{j=2}^N g_{1j}\sigma_j$.
Fix a smooth function $f$. For each $j =2,\ldots,N$, let
\[
h_j = \frac{1}{\sqrt{N}}\avg{\sigma_j f(\ell_1)}.
\]
Then we have 
\begin{equation}\label{ske1}
\sum_{j=2}^N g_{1j} h_j = \avg{\ell_1 f(\ell_1)}.
\end{equation}
On the other hand, an easy computation gives
\[
\fpar{h_j}{g_{1j}} = \frac{\avg{f'(\ell_1)} +\beta \avg{\sigma_1 f(\ell_1) } -\beta \avg{\sigma_j f(\ell_1)} \avg{\sigma_1\sigma_j}}{N}.
\]
Note that $\ell_1$ does not depend on $\sigma_1$, and the conditional expectation of $\sigma_1$ given $\sigma_2,\ldots,\sigma_N$ is $\tanh(\beta \ell_1)$. Thus,
\[
\avg{\sigma_1f(\ell_1)} = \avg{\tanh(\beta \ell_1) f(\ell_1)}.
\]
Again, it follows from the high temperature condition \eqref{high} for $\beta < 1$ and $h = 0$ that for $2\le j\le N$,
\[
\avg{\sigma_1\sigma_j} \approx \avg{\sigma_1}\avg{\sigma_j} = 0.
\]
Combining, we see that
\begin{equation}\label{ske2}
\sum_{j=2}^N \fpar{h_j}{g_{1j}} \approx \avg{f'(\ell_1)} + \beta \avg{\tanh(\beta \ell_1) f(\ell_1)}.
\end{equation}
Now, using integration by parts for Gaussian random variables, we get 
\[
\ee\biggl(\sum_{j=2}^N g_{1j}h_j\biggr) =\ee\biggl(\sum_{j=2}^N \fpar{h_j}{g_{1j}}\biggr).
\]
In view of \eqref{ske1} and \eqref{ske2}, this is equivalent to
\begin{equation}\label{annealed}
\ee\avg{\ell_1 f(\ell_1) - f'(\ell_1)-\beta \tanh(\beta \ell_1) f(\ell_1)} \approx 0.
\end{equation}
As noted in Subsection \ref{steinmix},
\[
T f(x) = xf(x)-f'(x)-\beta \tanh(\beta x) f(x)
\]
is a Stein characterizing operator for the mixture Gaussian density
$\frac{1}{2}\phi_{\beta, 1} + \frac{1}{2}\phi_{-\beta,1}$.
Note that this procedure `discovers' that the (averaged) limiting distribution of $\ell_1$ is the above Gaussian mixture.
%\pause
%\item The Approximation Lemma helps only in the high temperature phase.

\subsection{Quenched distributions and the Approximation Lemma}
In the above example, we sketched a derivation of the limiting unconditional (i.e.\ averaged over disorder) distribution for the local field at site $1$, essentially using Gaussian integration by parts. To prove the result for the quenched distribution, it does not suffice to show~\eqref{annealed}, but rather, we have to show 
\[
\avg{\ell_1 f(\ell_1) - f'(\ell_1)-\beta \tanh(\beta \ell_1) f(\ell_1)} \approx 0 \ \ \ \text{with high probability.}
\]
In other words, we have to show
\[
\sum_{j=2}^N g_{1j}h_j \approx \sum_{j=2}^N \fpar{h_j}{g_{1j}} \ \ \ \text{with high probability.}
\]
This is a recurring issue whenever we have to prove a quenched CLT. The following result, which we call the `approximation lemma', becomes our main tool. The proof of the lemma is so short that we present it right away.
\begin{lmm}\label{mainlmm}
Suppose $g = (g_1,\ldots,g_n)$ is a collection of  independent standard Gaussian random variables, and $h_1,\ldots,h_n$ are absolutely continuous functions of $g$. Assume that $h_i$ are elements of the Sobolev space $H^{1,2}$ with respect to the Gaussian measure on $\rr$.  Then
\begin{align*}
\ee\biggl(\sum_{j=1}^ng_j h_j - \sum_{j=1}^n \fpar{h_j}{g_j}\biggr)^2 = \sum_{j=1}^n\ee h_j^2 + \sum_{j,k=1}^n \ee\biggl(\fpar{h_j}{g_k} \fpar{h_k}{g_j}\biggr).
\end{align*}
\end{lmm}
\begin{proof}
By taking convolutions with smooth kernels, we can assume that $h_1,\ldots,h_n$ are twice continuously differentiable. Let
\[
h = \sum_{j=1}^n \biggl(g_j h_j - \fpar{h_j}{g_j}\biggr).
\]
Then
\[
\ee h^2 = \sum_{j=1}^n\ee\biggl(g_j h_j h -  \fpar{h_j}{g_j}h\biggr).
\]
Integration-by-parts gives 
\[
\ee(g_j h_j h) = \ee \biggl( \fpar{h_j}{g_j} h + h_j\fpar{h}{g_j}\biggr).
\]
Thus, 
\[
\ee h^2 = \sum_{j=1}^n \ee\biggl(h_j \fpar{h}{g_j}\biggr).
\]
Now 
\begin{align*}
\fpar{h}{g_j} &= h_j + \sum_{k=1}^n \biggl(g_k \fpar{h_k}{g_j} - \mpar{h_k}{g_j}{g_k}\biggr).
\end{align*}
Therefore,
\[
\ee h^2 = \sum_{j=1}^n \ee h_j^2 + \sum_{j,k=1}^n \ee\biggl(g_k h_j \fpar{h_k}{g_j} - h_j \mpar{h_k}{g_j}{g_k}\biggr).
\]
Again, using integration-by-parts, we see that
\[
\ee\biggl(g_k h_j \fpar{h_k}{g_j}\biggr) = \ee\biggl(\fpar{h_j}{g_k} \fpar{h_k}{g_j}+ h_j \mpar{h_j}{g_j}{g_k}\biggr).
\]
This completes the proof.
\end{proof}

\subsection{Other results}
The following theorems are some further examples of CLTs for the SK model that can be proved via Stein's method. In all cases, we obtain total variation error bounds. Although the bounds are  probably suboptimal, this is the only method available that can give such bounds.
\subsubsection{The cavity field}
Suppose $g_1,\ldots,g_N$ are i.i.d.\ standard Gaussian random variables, independent of the disorder $(g_{ij})_{i<j\le N}$. The `cavity field' $\ell$ is defined as
\begin{equation}\label{cavitydef}
\ell = \frac{1}{\sqrt{N}}\sum_{i=1}^N g_i \sigma_i.
\end{equation}
The name `cavity field' comes from the role played by $\ell$ in the cavity method for solving the SK model in the high temperature regime.
%The quantity $\ell$ is called the cavity field for the following reason: The interaction between $\sigma_1,\ldots,\sigma_N$ and $\sigma_{N+1}$ in the SK model with $N+1$ particles is precisely $\frac{1}{\sqrt{N}} \sum_{j=1}^N g_{j, N+1} \sigma_j$, and when this object is considered in the system with $\sigma_{N+1}$ removed (i.e., after creating a `cavity'), it is indistinguishable from $\ell$. The asymptotic gaussianity of cavity field is of central importance in the high temperature analysis of the SK model.
Note that the quenched average of $\ell$ is 
\[
\avg{\ell} = \frac{1}{\sqrt{N}}\sum_{i=1}^N g_i \avg{\sigma_i}.
\]
The following result states that under the Gibbs measure, $\ell$ is approximately Gaussian with mean $\avg{\ell}$ and variance $1-q$. The original proof of this result, without the error bound, can be found in Talagrand \cite{talagrand03}, page 87.
\begin{thm}\label{cavity}
Suppose the high temperature condition \eqref{high} is satisfied and $u: \rr \ra \rr$ is a bounded measurable function. Then
\begin{align*}
\ee\biggl(\avg{u(\ell)} - \int_{\rr} u(t) \phi_{\langle\ell\rangle,1-q}(t) dt\biggr)^2 &\le \frac{C(\beta,h)\|u\|_\infty^2}{\sqrt{N}},
\end{align*}
where $\phi$ is the Gaussian density defined in \eqref{normdens}.
\end{thm}

\subsubsection{The Hamiltonian}
Our next limit theorem is about the quantity
\begin{equation}\label{hamildef}
H := \frac{1}{N}\sum_{i<j\le N} g_{ij}\sigma_i\sigma_j - \frac{\sqrt{N}\beta}{2}.
\end{equation}
Note that this is just a linear transformation of the interaction term in the hamiltonian. We show that in the regime $\beta < 1$, $h=0$, the quenched law of $H$ is asymptotically Gaussian with mean $0$ and variance $1/2$. The case of general $\beta$ and $h$, even in the high temperature phase, seems to be much harder, and is currently under investigation.
\begin{thm}\label{hamil}
Let $H$ be defined as above. Suppose $\beta < 1$, $h=0$, and $u:\rr \ra \rr$ is a bounded measurable function. Then
\[
\ee\biggl(\avg{u(H)} - \int_{\rr} u(t)\phi_{0,1/2} (t) dt\biggr)^2 \le \frac{C(\beta)\|u\|_\infty^2}{N}.
\]
This gives a total variation error bound of order $1/\sqrt{N}$ in the central limit theorem for $H$.
\end{thm}
\noindent Again, this was originally proved in  Comets and Neveu \cite{cometsneveu95}, Proposition 5.2, albeit without an error bound.

\subsubsection{Quenched average of the spin at a site}
It is natural to ask about the limiting distribution of the random variables $(\langle \sigma_i\rangle)_{1\le i\le N}$. Although the joint distribution has no simple description, Talagrand proved that for any fixed $k$, the collection $(\langle\sigma_i \rangle)_{1\le i\le k}$ converges in law to  $(\tanh(\beta z_i\sqrt{q} + h))_{1\le i\le k}$, where $z_1,\ldots,z_k$ are independent standard Gaussian random variables (Theorem~2.4.12 in \cite{talagrand03}). 

Note that the term inside $\tanh$ in Corollary \ref{tap} is simply $\beta r_i + h$, with $r_i$ defined in \eqref{ri}. Hence, to compute the asymptotic distribution of $\langle \sigma_i \rangle$, it suffices to find out the limit law of $r_i$. The following result shows that $r_i$ is asymptotically Gaussian with mean $0$ and variance $q$, and gives a total variation error bound. 
\begin{thm}\label{newthm}
Suppose \eqref{high} holds for some $q>0$, and $r_i$ is defined as in~\eqref{ri}. Let $z$ be a standard Gaussian random variable. Then for any bounded measurable $u:\rr\ra\rr$, 
\[
\bigl|\ee u(r_i) - \ee u(z\sqrt{q})\bigr|\le \frac{C(\beta,h)\|u\|_\infty }{N^{1/4}}.
\]
Note that by Corollary \ref{tap}, this shows that $\langle \sigma_i\rangle$ is asymptotically distributed as $\tanh(\beta z\sqrt{q} + h)$. 
\end{thm}
\noindent The rest of the paper is organized as follows. Since the complete proofs involve some heavy computations, we give brief sketches of the proofs in the next section. The details are in Section \ref{proofs}. Section \ref{proofs} also contains a development of Stein's method for mixture Gaussian densities.

\section{Proof outlines}
In this section, we sketch the proofs of the theorem from Section \ref{intro} in the order of difficulty.
\vskip.1in
\noindent {\it Sketch of the proof of Theorem \ref{cavity}.} Recall the definition \eqref{cavitydef} of the cavity field $\ell$ and let $r = \langle\ell\rangle$.  Take any smooth function $f:\rr^2 \ra \rr$ and for each $j$, let
\[
h_j := \frac{1}{\sqrt{N}} \avg{\bigl(\sigma_j - \avg{\sigma_j}\bigr) f(\ell, r)}.
\]
Then
\[
\sum_{j=1}^N g_j h_j  =  \avg{(\ell -r) f(\ell, r)}.
\]
A careful calculation shows that under \eqref{high}, the approximation lemma can be applied, and 
\[
\sum_{j=1}^N \fpar{h_j}{g_j} \approx (1-q) \bigavg{\fpar{f}{x}(\ell,r)}.
\]
Combining, we get that for any smooth $f:\rr^2 \ra \rr$, 
\[
\bigavg{\frac{\ell - r}{1-q} f(\ell,r) - \fpar{f}{x}(\ell,r)} \approx 0.
\]
This shows that the law of $\ell$ under the Gibbs measure approximately satisfies the characterizing equation for the Gaussian law with mean $r$ and variance $1-q$. The proof can now be completed by standard techniques from Stein's method.
\vskip.1in
\noindent {\it Sketch of the proof of Theorem \ref{hamil}.}
Recall the definition \eqref{hamildef} of the centered hamiltonian $H$, and take any smooth function $f:\rr\ra\rr$. For each $i< j\le N$ let
\[
h_{ij} = \frac{1}{N}\avg{\sigma_i\sigma_j f(H)}.
\]
Then
\[
\sum_{i<j\le N} g_{ij} h_{ij} = \bigavg{\biggl(H + \frac{\sqrt{N}\beta}{2}\biggr) f(H)}.
\]
In the regime $\beta < 1$, $h = 0$, it is known that $R_{12} = O(N^{-1/2})$. Using this fact and some calculations, it follows that the approximation lemma can be applied to the collection $(g_{ij}, h_{ij})_{i<j\le N}$, and also that
\[
\sum_{i<j} \fpar{h_{ij}}{g_{ij}} \approx \frac{1}{2}\avg{f'(H)} + \frac{\sqrt{N}\beta}{2}\avg{f(H)}.
\]
This shows that for any smooth $f$,
\[
\avg{ H f(H) - {\textstyle\frac{1}{2}} f'(H)} \approx 0,
\]
and Stein's method does the rest.
\vskip.1in
\noindent{\it Sketch of the proof of Theorem \ref{mainthm}.} 
Recall the definitions \eqref{li} and \eqref{ri} of $\ell_i$ and $r_i$. It suffices to prove the theorem for $i = 1$. Take any smooth $f:\rr^2 \ra \rr$, and for each $2\le j\le N$, let
\[
h_j(g) =\frac{1}{\sqrt{N}} \avg{\bigl(\sigma_j - \avg{\sigma_j}\bigr)f(\ell_1,r_1)}.
\]
Then 
\[
\sum_{j=2}^N g_{1j} h_j = \avg{(\ell_1 - r_1) f(\ell_1,r_1)}.
\]
Now $h_j$ depends not only on $(g_{1j})_{2\le j\le N}$ but also on $(g_{ij})_{2\le i<j\le N}$. However, we can condition on $(g_{ij})_{2\le i<j\le N}$ and then apply the approximation lemma to show that
\[
\sum_{j=2}^N g_{1j} h_j \approx \sum_{j=2}^N \fpar{h_j}{g_{1j}}.
\]
In a number of steps, one can show that under \eqref{high},
\[
\sum_{j=2}^N \fpar{h_j}{g_{1j}} \approx (1-q) \bigavg{\beta \tanh(\beta \ell_1 + h) f(\ell_1,r_1) + \fpar{f}{x}(\ell_1,r_1)}.
\]
Combining it follows that for any smooth $f:\rr^2 \ra \rr$,
\[
\bigavg{\biggl(\frac{\ell_1-r_1}{1-q} - \beta \tanh(\beta \ell_1 + h)\biggr) f(\ell_1,r_1) - \fpar{f}{x}(\ell_1,r_1)} \approx 0.
\]
It turns out that an exact equality in the above equation characterizes the distribution $\nu_1$ from Theorem \ref{mainthm}. The proof can now be completed by Stein's method.
\vskip.1in
\noindent{\it Sketch of the proof of Theorem \ref{newthm}.} 
Again, it suffices to prove the theorem just for $r_1$. By a series of steps involving integration by parts and applications of the high temperature condition \eqref{high}, one can show that for any smooth $f:\rr\ra \rr$,
\begin{equation}\label{rcond}
\ee(r_1 f(r_1)) \approx q \ee \bigl(f'(r_1)\avg{\eta_1}\bigr),
\end{equation}
where
\[
\eta_1 = 1 + \frac{\beta\sigma_1}{\sqrt{N}}\sum_{j=2}^Ng_{1j}\bigl(\sigma_j - \avg{\sigma_j}\bigr) - \beta^2(1-q)\bigl(1 - \avg{\sigma_1}\sigma_1\bigr),
\]
Now let
\[
h_j = \frac{\beta\avg{\bigl(\sigma_j - \avg{\sigma_j}\bigr) \sigma_1}}{\sqrt{N}}.
\]
Clearly,
\begin{equation}\label{etacond}
\avg{\eta_1} = 1 + \sum_{j=2}^N g_{1j} h_j - \beta^2(1-q)\bigl(1-\avg{\sigma_1}^2\bigr).
\end{equation}
A series of steps using the high temperature condition \eqref{high} give
\[
\sum_{j=2}^N \fpar{h_j}{g_{1j}} \approx \beta^2(1-q)\bigl(1-\avg{\sigma_1}^2\bigr).
\]
Applications of \eqref{high} also imply that the approximation lemma can be used in this case to deduce that
\[
\sum_{j=2}^N g_{1j}h_j \approx \sum_{j=2}^N \fpar{h_j}{g_{1j}},
\]
and therefore the last two terms in \eqref{etacond} approximately cancel each other out, leaving us with the conclusion that $\langle \eta_1\rangle \approx 1$. Combining with \eqref{rcond}, we see that for any smooth $f:\rr \ra \rr$, $\ee (r_1 f(r_1)) \approx q \ee(f'(r_1))$. The proof is now completed by Stein's method.

\section{Complete proofs}\label{proofs}
\subsection{Some estimates}
Applying Lemma \ref{mainlmm} in our problems require a substantial amount of computation. The purpose of this subsection is to organize the computations into a friendly and accessible system.

In this subsection and the rest of the manuscript, we switch to the convention that $C(\beta,h)$ denotes any constant that depends only on $\beta$ and $h$. In particular, the value of $C(\beta,h)$ may change from line to line. 

%In this subsection we work out some technical estimates required for applying Lemma \ref{mainlmm}.  Throughout this section and the rest of the manuscript, $C(\beta,h)$ will denote any constant that depends only on $\beta$ and $h$.

Let us first recall our conventions. Configurations chosen independently given the disorder are denoted by $\sigma^1, \sigma^2$, etc.  The overlap between $\sigma^1$ and $\sigma^2$ is defined as
\[
R_{12} = \frac{1}{N}\sum_{i=1}^N \sigma_i^1 \sigma_i^2.
\]
%The {\it local field} at site $1$ is the total interaction of the particle at $1$ with the rest of the particles, defined as
%\[
%\ell_1 = \ell_1(g, \sigma) = \frac{1}{\sqrt{N}} \sum_{j=2}^N g_{1j} \sigma_j.
%\]
Recall that we have a number $q$, depending on $\beta$ and $h$, such that
\begin{equation}\label{rconv}
\ee\avg{(R_{12}-q)^4} \le \frac{C(\beta,h)}{N^2}.
\end{equation}
%The function $r_1$ of the disorder $g$ is defined as
%\[
%r_1 = r_1(g) = \frac{1}{\sqrt{N}} \sum_{j=2}^N g_{1j} \avg{\sigma_j} - \beta(1-q)\avg{\sigma_1}.
%\]
Let us begin our computations with the following straightforward formula: For any function $v = v(g, \sigma)$ of the disorder $g$ and the configuration $\sigma$, and any $i,j$, we have
\begin{equation}\label{deriv1}
\begin{split}
\fpar{\avg{v}}{g_{ij}} &= \bigavg{\fpar{v}{g_{ij}}} + \frac{\beta}{\sqrt{N}}\avg{v\bigl(\sigma_i\sigma_j - \avg{\sigma_i\sigma_j}\bigr)}  \\
&= \bigavg{\fpar{v}{g_{ij}}} + \frac{\beta}{\sqrt{N}}\avg{\big(v-\avg{v}\bigr)\sigma_i\sigma_j}.
\end{split}
\end{equation}
For each $j$, let
\begin{equation}\label{xidef}
\xxi_j = \xxi_j(g,\sigma) = \sigma_j - \avg{\sigma_j}.
\end{equation}
Then by \eqref{deriv1},
\begin{equation}\label{xider}
\fpar{\xxi_j}{g_{kl}} = -\fpar{\avg{\sigma_j}}{g_{kl}} = -\frac{\beta}{\sqrt{N}} \avg{\xxi_j\sigma_k\sigma_l}.
\end{equation}
Now let $v(g,\sigma)$ be a bounded function of $g$ and $\sigma$. Then
\begin{align*}
\frac{1}{N}\sum_{j=1}^N\ee\avg{\xxi_jv}^2 &= \frac{1}{N}\sum_{j=1}^N \ee\avg{\bigl(\sigma_j^1 - \avg{\sigma_j}\bigr)\bigl(\sigma_j^2 - \avg{\sigma_j}\bigr)v(g,\sigma^1)v(g,\sigma^2)}\\
&= \ee\avg{\bigl(R_{12} - R_{14} - R_{23}+ R_{34}\bigr)v(g,\sigma^1) v(g,\sigma^2)}.
\end{align*}
From this and the inequality \eqref{rconv}, we get
\begin{equation}\label{est1}
\frac{1}{N}\sum_{j=1}^N \ee\avg{\xxi_j v}^2 \le \frac{C(\beta,h)\bigl(\ee\avg{v^4}\bigr)^{1/2}}{\sqrt{N}}.
\end{equation}
Next, note that
\begin{align*}
\frac{1}{N}\sum_{j=1}^N\avg{\xxi_j\sigma_j v}&= \frac{1}{N} \sum_{j=1}^N\bigl(\avg{v(g,\sigma)} - \avg{\sigma_j^1\sigma_j^2 v(g, \sigma^1)}\bigr)\\
&= \avg{(1-R_{12})v(g,\sigma^1)}.
\end{align*}
Thus, we have
\begin{equation}\label{est2}
\ee\biggl(\frac{1}{N}\sum_{j=1}^N\avg{\xxi_j\sigma_j v} - (1-q)\avg{v}\biggr)^2 \le \frac{C(\beta,h) \bigl(\ee\avg{v^4}\bigr)^{1/2}}{N}.
\end{equation}
If $w$ is another function of $g$ and $\sigma$, then
\begin{align*}
\frac{1}{N}\sum_{j=1}^N\avg{\sigma_j v}\avg{\sigma_j w} &= \avg{R_{12}v(g,\sigma^1)w(g,\sigma^2)}.
\end{align*}
Thus, we have
\begin{equation}\label{est30}
\ee\biggl(\frac{1}{N}\sum_{j=1}^N\avg{\sigma_j v}\avg{\sigma_j w} - q\avg{v}\avg{w}\biggr)^2 \le \frac{C(\beta,h)\bigl(\ee\avg{v^4}\avg{w^4}\bigr)^{1/2}}{N}.
\end{equation}
Next, note that
\begin{align*}
\frac{1}{N}\sum_{j=1}^N\avg{\xxi_j v}\avg{\sigma_j w} &= \frac{1}{N}\sum_{j=1}^N\bigl(\avg{\sigma_j v}\avg{\sigma_j w} - \avg{\sigma_j} \avg{v} \avg{\sigma_j w}\bigr) \\
&= \avg{(R_{12}- R_{13})v(g,\sigma^1)w(g,\sigma^2)}. 
\end{align*}
This shows that
\begin{equation}\label{est3}
\frac{1}{N}\sum_{j=1}^N\ee\bigl( \avg{\xxi_j v}\avg{\sigma_j w} \bigr) \le \frac{C(\beta,h)\bigl(\ee\avg{v^2}\avg{w^2}\bigr)^{1/2}}{\sqrt{N}},
\end{equation}
and moreover
\begin{equation}\label{est31}
\ee\biggl(\frac{1}{N}\sum_{j=1}^N \avg{\xxi_j v}\avg{\sigma_j w} \biggr)^2 \le \frac{C(\beta,h)\bigl(\ee\avg{v^4}\avg{w^4}\bigr)^{1/2}}{N}.
\end{equation}
The inequality \eqref{est3} readily implies the following important lemma.
\begin{lmm}\label{jkbd}
Let $v_1,\ldots,v_N, w_1,\ldots, w_N$ be arbitrary functions of $g$ and $\sigma$. Then we have 
\[
\ee\biggl(\frac{1}{N^2}\sum_{j,k=1}^N \avg{\xxi_j v_k}\avg{w_k \sigma_j}\biggr) \le \frac{C(\beta,h)}{N^{3/2}} \sum_{k=1}^N \bigl(\ee\avg{v_k^2}\avg{w_k^2}\bigr)^{1/2}.
\]
\end{lmm}
\noindent The above result is generally used as follows. Given functions $f_2,\ldots,f_N$ of the disorder $g$, we find $v_j$ and $w_j$ such that
\[
\fpar{f_j}{g_{1k}} = \frac{\avg{\xxi_j v_k}}{N} = \frac{\avg{w_j \sigma_k}}{N},
\]
and apply the bound from Lemma \ref{jkbd} to extract information from Lemma~\ref{mainlmm}.

The next lemma is necessary for bounding the  moments of functions of $(g,\sigma)$ that arise when we try to apply the inequalities derived above.
\begin{lmm}\label{expan}
Let $b_1(\sigma),\ldots,b_m(\sigma)$ be arbitrary functions of $\sigma$, taking values in the interval $[-1,1]$. Then for any $k\ge 1$ and any distinct collection of indices $2\le j_1,\ldots, j_k\le N$, we have
\begin{align*}
\ee\bigl(g_{1j_1}g_{1j_2}\cdots g_{1j_k}\avg{b_1}\avg{b_2}\cdots \avg{b_m}\bigr)&\le \frac{C(m, k)\beta^k}{N^{k/2}},
\end{align*}
where $C(m,k)$ is a constant depending only on $m$ and $k$.
\end{lmm}
\begin{proof}
Let us use induction on $k$. For $k=1$, observe that by integration-by-parts and \eqref{deriv1}, we have
\begin{align*}
&\ee\bigl(g_{1j_1}\avg{b_1}\avg{b_2}\cdots \avg{b_m}\bigr) = \sum_{l = 1}^m \ee \biggl(\fpar{\avg{b_l}}{g_{1j_1}} \prod_{l'\ne l} \avg{b_{l'}}\biggr)\\
&= \frac{\beta}{\sqrt{N}}\sum_{l = 1}^m \ee \biggl(\avg{\bigl(b_l - \avg{b_l}\bigr)\sigma_1\sigma_{j_1}} \prod_{l'\ne l} \avg{b_{l'}}\biggr) \le \frac{C(m)\beta}{\sqrt{N}}.
\end{align*}
Now assume that the result is true up to $k-1$ (and any $m$). Again, using integration-by-parts and \eqref{deriv1}, we have
\begin{align*}
&\ee\bigl(g_{1j_1}g_{1j_2}\cdots g_{1j_k}\avg{b_1}\avg{b_2}\cdots \avg{b_m}\bigr)\\
&\le \frac{\beta}{\sqrt{N}}\sum_{l = 1}^m \ee \biggl(g_{1j_2}\cdots g_{1j_k}\fpar{\avg{b_l}}{g_{1j_1}} \prod_{l'\ne l} \avg{b_{l'}}\biggr)\\
&= \frac{\beta}{\sqrt{N}}\sum_{l = 1}^m \ee \biggl(g_{1j_2}\cdots g_{1j_k}\avg{\bigl(b_l - \avg{b_l}\bigr)\sigma_1\sigma_{j_1}} \prod_{l'\ne l} \avg{b_{l'}}\biggr).
\end{align*}
A straightforward application of the induction hypothesis for $k-1$ completes the proof.
\end{proof}
\noindent The following function will appear several times in the sequel. 
\begin{equation}\label{l1p}
\dot{\ell}_1 =\dot{\ell}_1(g,\sigma) = \frac{1}{\sqrt{N}} \sum_{j=2}^N g_{1j} \xxi_j.
\end{equation}
Take any $k \ge 1$. A simple application of Lemma \ref{expan} to each term in the expansion of $\dot{\ell}_1^k$ shows that
\begin{equation}\label{est4}
\ee(\dot{\ell}_1^k) \le C(\beta, k).
\end{equation}
The important thing is that the bound does not depend on $N$.

\subsection{Proof of Theorem \ref{cavity}}
We will continue using the notation introduced in the previous subsections. Let us briefly recall the setting. 
Suppose $g_1,\ldots,g_N$ are i.i.d.\ standard Gaussian random variables, independent of $(g_{ij})_{i<j\le N}$. The cavity field is defined as
\[
\ell := \frac{1}{\sqrt{N}}\sum_{j=1}^N g_j \sigma_j.
\]
Our objective is to show that under the Gibbs measure, $\ell$ is approximately distributed as a Gaussian random variable with mean 
\[
r := \frac{1}{\sqrt{N}}\sum_{j=1}^N g_j\avg{\sigma_j}
\]
and variance $1 - q$.

Take any bounded measurable function $u: \rr \ra [-1,1]$ and suppose $f:\rr^2\ra \rr$ is a solution to
\[
\fpar{f}{x}(x,y) - \frac{x-y}{1-q} f(x,y) = u(x) - \int_{\rr} u(t) \phi_{y,1-q}(t) dt.
\]
For simplicity, we let $f_1$ and $f_2$ denote $\fpar{f}{x}$ and $\fpar{f}{y}$. From Lemma \ref{steinlmm} it follows that $|f|$, $|f_1|$ and $|f_2|$ are uniformly  bounded by $C(\beta,h)$. For each $j$, let
\[
h_j := \frac{1}{\sqrt{N}} \avg{\xxi_j f(\ell, r)}.
\]
Then
\begin{equation}\label{ineq141}
\sum_{j=1}^N g_j h_j  =  \avg{(\ell -r) f(\ell, r)}.
\end{equation}
In the rest of the proof, we will simply write $f$, $f_1$ and $f_2$ instead of $f(\ell,r)$, etc.
Note that for any $j,k$,
\begin{align*}
\fpar{h_j}{g_k} &= \frac{1}{N}\avg{\xxi_j \sigma_k f_1} + \frac{1}{N}\avg{\xxi_j f_2}\avg{\sigma_k}.
\end{align*}
Thus, putting 
\begin{align*}
v_k = \sigma_k f_1 + f_2 \avg{\sigma_k} \ \text{ and } \ w_j = \xxi_j f_1 + \avg{\xxi_j f_2},
\end{align*}
we see that
\[
\fpar{h_j}{g_k} = \frac{\avg{\xxi_j v_k}}{N} = \frac{\avg{w_j \sigma_k}}{N}.
\]
Hence by Lemma \ref{jkbd}, we have
\[
\sum_{j,k=1}^N \ee\biggl(\fpar{h_j}{g_k} \fpar{h_k}{g_j}\biggr) \le \frac{C(\beta,h)}{\sqrt{N}}.
\]
Again, from \eqref{est1} we have 
\[
\sum_{j=1}^N \ee(h_j^2) \le \frac{C(\beta,h)}{\sqrt{N}}.
\]
Combining and applying Lemma \ref{mainlmm}, we get
\begin{equation}\label{ineq142}
\ee\biggl(\sum_{j=1}^N g_j h_j - \sum_{j=1}^N \fpar{h_j}{g_j}\biggr)^2 \le \frac{C(\beta,h)}{\sqrt{N}}.
\end{equation}
Again, note that
\[
\sum_{j=1}^N \fpar{h_j}{g_j} = \frac{1}{N}\sum_{j=1}^N \bigl(\avg{\xxi_j \sigma_j f_1} + \avg{\xxi_j f_2}\avg{\sigma_j}\bigr).
\]
By \eqref{est2} and \eqref{est31}, this gives
\begin{equation}\label{ineq143}
\ee\biggl(\sum_{j=1}^N \fpar{h_j}{g_j} - (1-q)\avg{f_1}\biggr)^2 \le \frac{C(\beta,h)}{N}.
\end{equation}
Combining \eqref{ineq141}, \eqref{ineq142} and \eqref{ineq143}, we finally get
\begin{align*}
\ee\biggl(\avg{u(\ell)} - \int_{\rr} u(t) \phi_{r,1-q}(t) dt\biggr)^2 &= \frac{1}{(1-q)^2}\ee\avg{ (1-q)f_1 - (\ell - r) f}^2 \\
&\le \frac{C(\beta,h)}{\sqrt{N}}.
\end{align*}
This completes the proof of Theorem \ref{cavity}.

\subsection{Proof of Theorem \ref{hamil}}
Recall that the centered hamiltonian $H$ was defined  as
\[
H := \frac{1}{N}\sum_{i<j\le N} g_{ij}\sigma_j \sigma_j - \frac{\sqrt{N}\beta}{2}.
\]
Take any $u:\rr \ra [-1,1]$, and let $f$ be a solution to
\[
f'(x) - 2x f(x) = u(x) - \int_{\rr} u(t) \phi_{0,1/2}(t) dt.
\]
Again by Lemma \ref{steinlmm}, $|f|$ and $|f'|$ are uniformly bounded by $C(\beta,h)$. 
For each $i,j$, let 
\[
h_{ij} = \frac{1}{N}\avg{\sigma_i\sigma_j f(H)}.
\]
In the following, we will write $f$ and $f'$ for $f(H(\sigma))$ and $f'(H(\sigma))$ for notational convenience. When we have expressions involving multiple replicas, $f$ will stand for  $f(H(\sigma^1))$. We have 
\begin{equation}\label{hhg}
\fpar{h_{ij}}{g_{kl}} = \frac{1}{N^2}\avg{\sigma_i\sigma_j\sigma_k\sigma_lf'} + \frac{\beta}{N^{3/2}}\bigl( \avg{\sigma_i\sigma_j \sigma_k \sigma_l f} - \avg{\sigma_i\sigma_j f}\avg{\sigma_k\sigma_l}\bigr).
\end{equation}
Using identities like
\begin{align*}
&\frac{1}{N^4}\sum_{i,j,k,l} \avg{\sigma_i\sigma_j\sigma_k\sigma_l f}^2 = \avg{R_{12}^4 f}, \\
&\frac{1}{N^4}\sum_{i,j,k,l} \avg{\sigma_i\sigma_j\sigma_k\sigma_l f}\avg{\sigma_i\sigma_j f}\avg{\sigma_k \sigma_l} = \avg{R_{12}^2 R_{13}^2 f}, \ \ \text{etc.,}
\end{align*}
we get
\[
\sum_{i<j,\ k< l} \ee\biggl(\fpar{h_{ij}}{g_{kl}} \fpar{h_{kl}}{g_{ij}}\biggr) \le C(\beta) N \ee\avg{R_{12}^4} \le \frac{C(\beta)}{N}.
\]
Similarly,
\[
\sum_{i<j} \ee(h_{ij}^2) \le C(\beta) \ee\avg{R_{12}^2} \le \frac{C(\beta)}{N}.
\]
Another similar verification starting from the formula \eqref{hhg} shows that
\[
\sum_{i<j} \fpar{h_{ij}}{g_{ij}} = \frac{1}{2}\avg{f'} + \frac{\sqrt{N}\beta}{2}\avg{f} + \mathcal{R},
\]
where $\mathcal{R}$ is a remainder term satisfying 
\[
\ee(\mathcal{R}^2)\le \frac{C(\beta)}{N}.
\]
The proof is now completed by applying Lemma \ref{mainlmm}.

\subsection{Proof of Theorem \ref{mainthm}}
It suffices to prove the result for $i=1$. Note that $\nu_1$ is simply the probability distribution $M(\beta,h, r_1,1-q)$. 
Without loss of generality, we can assume that $\|u\|_\infty \le 1$. Suppose $f:\rr^2 \ra \rr$ is a solution of the differential equation
\begin{equation}\label{fdef}
\begin{split}
&\fpar{f}{x}(x,y) - \biggl(\frac{x-y}{1-q} - \beta \tanh(\beta x + h)\biggr) f(x,y) \\
&= u(x) - \int_{\rr} u(t) \psi_{\beta, h, y,1-q}(t) dt.
\end{split}
\end{equation}
By Lemma \ref{steinlmm}, such an $f$ exists and moreover, we can guarantee that $|f|$, $\bigl|\fpar{f}{x}\bigr|$, and $\bigl|\fpar{f}{y}\bigr|$ are all bounded by $C(\beta,h)$. As before, to lighten notation, we let $f_1$ and $f_2$ denote the two partial derivatives of $f$. We have to prove that for any $i$, 
\[
\ee\biggl(\avg{u(\ell_i)} - \int_{\rr} u(x) \psi_{\beta, h, r_i, 1-q} dx\biggr)^2 \le \frac{C(\beta,h) \|u\|_\infty^2}{\sqrt{N}},
\]
where 
\begin{align*}
r_i &= r_i(g) := \frac{1}{\sqrt{N}} \sum_{j\ne i} g_{ij} \avg{\sigma_j} - \beta(1-q)\avg{\sigma_i}.
\end{align*}
By the definition and properties of $f$, this is clearly equivalent to proving that
\[
\ee\bigavg{f_1(\ell_i,r_i) - \biggl(\frac{\ell_i - r_i}{1-q} - \beta \tanh(\beta \ell_i + h)\biggr) f(\ell_i,r_i)}^2 \le \frac{C(\beta,h)}{\sqrt{N}},
\]
and this is what we aim to prove in the next few pages. Note that it suffices to fix $i=1$.  
Recall that we defined
\[
\xxi_j := \sigma_j - \avg{\sigma_j}. 
\]
For each $j\ge 2$, let
\[
h_j(g) =\frac{1}{\sqrt{N}} \avg{\xxi_jf(\ell_1, r_1)}, 
\]
where recall that
\[
\ell_1 = \ell_1(g, \sigma) = \frac{1}{\sqrt{N}} \sum_{j=2}^N g_{1j} \sigma_j
\]
and
\[
r_1 = r_1(g) = \frac{1}{\sqrt{N}} \sum_{j=2}^N g_{1j} \avg{\sigma_j} - \beta(1-q)\avg{\sigma_1}.
\]
In what follows, the random variable $f(\ell_1,r_1)$  is simply denoted by $f$ to lighten notation. The distinction between the random variable $f$ and the function $f$ should be clear from the context. Similar remarks apply to $f_1$ and $f_2$ also.

Now for any $j,k\ge 2$, simple applications of \eqref{deriv1} and \eqref{xider} gives 
\begin{align}
\fpar{h_j}{g_{1k}} &= \frac{1}{\sqrt{N}}\bigavg{\fpar{\xxi_j}{g_{1k}} f}  + \frac{1}{\sqrt{N}}\bigavg{\xxi_j f_1 \fpar{\ell_1}{g_{1k}}} + \frac{1}{\sqrt{N}} \avg{\xxi_j f_2} \fpar{r_1}{g_{1k}}\nonumber \\
&\quad + \frac{\beta}{N}\avg{(\xxi_j f)\bigl(\sigma_1\sigma_k - \avg{\sigma_1\sigma_k}\bigr)}\nonumber \\
&= - \frac{\beta}{N} \avg{\xxi_j \sigma_1\sigma_k}\avg{f} +  \frac{1}{N}\avg{\xxi_j f_1 \sigma_k} + \frac{1}{\sqrt{N}} \avg{\xxi_j f_2} \fpar{r_1}{g_{1k}} \label {hder} \\
&\quad + \frac{\beta}{N}\avg{(\xxi_j f)\bigl(\sigma_1\sigma_k - \avg{\sigma_1\sigma_k}\bigr)}.\nonumber
\end{align}
A further use of \eqref{deriv1} and \eqref{xider} gives
\[
\fpar{r_1}{g_{1k}} = \frac{\avg{\sigma_k}}{\sqrt{N}} + \frac{\beta}{N}\sum_{l\ge 2} g_{1l}\avg{\xxi_l \sigma_1\sigma_k} - \frac{\beta^2(1-q)}{\sqrt{N}} \avg{\xxi_1\sigma_1\sigma_k}.
\]
Recalling the definition \eqref{l1p} of $\dot{\ell}_1$ and putting
\begin{equation}\label{eta1def}
\eta_1 := 1 + \beta\dot{\ell}_1\sigma_1 - \beta^2(1-q)\xxi_1\sigma_1,
\end{equation}
we see that
\begin{equation}\label{r1deriv}
\begin{split}
\fpar{r_1}{g_{1k}} &= \frac{\avg{\eta_1 \sigma_k}}{\sqrt{N}}.
\end{split}
\end{equation}
Thus, putting
\begin{align*}
v_k &= -\beta \sigma_1 \sigma_k\avg{f}  + f_1\sigma_k  + f_2\avg{\eta_1\sigma_k} +\beta f\sigma_1\sigma_k - \beta f\avg{\sigma_1\sigma_k}, \ \ \text{and}\\
w_j &= -\beta \xxi_j \sigma_1 \avg{f}  + \xxi_j f_1 + \avg{\xxi_j f_2}\eta_1 + \beta \xxi_j f \sigma_1 - \beta \avg{\xxi_j f}\sigma_1,
\end{align*}
and organizing the terms in \eqref{hder}, we get
\[
\fpar{h_j}{g_{1k}} = \frac{\avg{\xxi_j v_k}}{N} = \frac{\avg{w_j \sigma_k}}{N}.
\]
Since $f$, $f_1$, and $f_2$ are uniformly bounded by $C(\beta,h)$ and $\ee\avg{\dot{\ell}_1^4} \le C(\beta,h)$ by \eqref{est4}, an application of Lemma \ref{jkbd} gives
\begin{align}\label{term2}
\ee\biggl(\sum_{j,k=2}^N \fpar{h_j}{g_{1k}}\fpar{h_k}{g_{1j}}\biggr) &= \ee\biggl(\frac{1}{N^2} \sum_{j,k=2}^N \avg{\xxi_j v_k} \avg{w_k \sigma_j}\biggr) \le \frac{C(\beta,h)}{\sqrt{N}}.
\end{align}
Again, since $f$ is bounded by $C(\beta,h)$, we can use \eqref{est1} to get
\begin{equation}\label{term1}
\sum_{j=2}^N\ee(h_j^2) \le \frac{C(\beta,h)}{\sqrt{N}}.
\end{equation}
Applying Lemma \ref{mainlmm}, using the bounds \eqref{term1} and \eqref{term2} obtained above, we finally get
\begin{equation}\label{bd1}
\begin{split}
\ee\biggl(\sum_{j=2}^Ng_{1j} h_j - \sum_{j=2}^N \fpar{h_j}{g_{1j}}\biggr)^2 &= \sum_{j=2}^N\ee(h_j^2) + \sum_{j,k=2}^N \ee\biggl(\fpar{h_j}{g_{1k}}\fpar{h_k}{g_{1j}}\biggr) \\
&\le \frac{C(\beta,h)}{\sqrt{N}}.
\end{split}
\end{equation}
Note that although $h_j$ is a function of the whole collection $(g_{jk})_{1\le j<k\le N}$, we can first condition on $(g_{jk})_{2\le j < k\le N}$ and apply Lemma \ref{mainlmm}, and then take the unconditional expectation to get the first line in \eqref{bd1}.

Now let us define
\[
\gamma_1 := -\beta \sigma_1 \avg{f} + f_1 +\beta \sigma_1 f.
\]
Then from the expressions \eqref{hder} and \eqref{r1deriv} we see that
\[
\sum_{j=2}^N \fpar{h_j}{g_{1j}} = \frac{1}{N}\sum_{j=2}^N \biggl(\avg{\xxi_j \sigma_j \gamma_1} + \avg{\xxi_j f_2} \avg{\eta_1 \sigma_j} - \beta \avg{\xxi_j f}\avg{\sigma_1 \sigma_j}\biggr)
\]
Applying \eqref{est2} for the first term and \eqref{est31} for the second and third terms, we~have 
\begin{align}\label{bd2}
\ee\biggl(\sum_{j=2}^N\fpar{h_j}{g_{1j}} - (1-q)\avg{\gamma_1}\biggr)^2 &\le \frac{C(\beta,h)}{N}.
\end{align}
Note the most crucial point in the derivation of \eqref{bd1} and \eqref{bd2} is that by Lemma \ref{steinlmm}, the bounds on $f$ and its derivatives  depend only on $(\beta,h)$. The parameter $\mu$ in the mixture Gaussian distribution, which equals $r_1$ in this proof, does not actually behave as a fixed parameter because we have defined $f$ as a function of two variables, one of which is $\mu$. This is why we need to have $f$ defined on $\rr^2$ instead of $\rr^1$.

Now $f = f(\ell_1,r_1)$ does not depend on $\sigma_1$. Also recall that under the Gibbs measure, the conditional expectation of $\sigma_1$ given $\sigma_2,\ldots, \sigma_N$ is simply $\tanh(\beta \ell_1 + h)$. Thus,
\[
\avg{\sigma_1f} = \avg{\tanh(\beta \ell_1 + h) f}.
\]
Using the above identity to compute $\avg{\gamma_1}$, we see that
\begin{align*}
&\sum_{j=2}^N g_{1j}h_j - (1-q)\avg{\gamma_1} \\
&= \avg{\bigl(\ell_1-r_1 - \beta(1-q) \tanh(\beta \ell_1 + h)\bigr) f} - (1-q)\avg{f_1}.
\end{align*}
Combining \eqref{bd1} and \eqref{bd2}, and dividing by $1-q$ throughout, we get
\[
\ee\bigavg{f_1 - \biggl(\frac{\ell_1 - r_1}{1-q} - \beta \tanh(\beta \ell_1 + h)\biggr) f}^2 \le \frac{C(\beta,h)}{\sqrt{N}}.
\]
But by the definition \eqref{fdef} of $f$, 
\begin{align*}
&f_1(\ell_1,r_1) - \biggl(\frac{\ell_1 - r_1}{1-q} - \beta \tanh(\beta \ell_1 + h)\biggr) f(\ell_1,r_1) \\
&= u(\ell_1) - \int_{\rr} u(x) \psi_{\beta, h, r_1,1-q}(x) dx.
\end{align*}
This completes the proof of Theorem \ref{mainthm}.

\subsection{Proof of Theorem \ref{newthm}}
Recall the definitions
\[
r_1 :=  \frac{1}{\sqrt{N}} \sum_{j=2}^N g_{1j} \avg{\sigma_j} - \beta(1-q)\avg{\sigma_1}
\]
and 
\begin{equation*}
\eta_1 := 1 + \beta\dot{\ell}_1\sigma_1 - \beta^2(1-q)\xxi_1\sigma_1.
\end{equation*}
Let $u:\rr \ra [-1,1]$ be a measurable map. By Lemma \ref{steinlmm}, there exists an absolutely continuous function $f:\rr \ra \rr$ such that
\[
f'(x)-\frac{x}{q}f(x) = u(x) - \ee u(z\sqrt{q}),
\]
where $z$ is a standard Gaussian random variable, and moreover $|f|$ and $|f'|$ can be uniformly bounded by $C(\beta,h)$. From the definition \eqref{ri} of $r_1$, we see that
\begin{equation}\label{step1}
r_1 f(r_1) = \frac{1}{\sqrt{N}} \sum_{j=2}^N g_{1j} \avg{\sigma_j} f(r_1) - \beta(1-q) \avg{\sigma_1} f(r_1).
\end{equation}
Now, by integration-by-parts and the identities \eqref{xider} and \eqref{r1deriv}, we have
\begin{align*}
\ee\bigl(g_{1j} \avg{\sigma_j} f(r_1)\bigr) &= \ee\biggl(\fpar{\avg{\sigma_j}}{g_{1j}} f(r_1) + \avg{\sigma_j} f'(r_1) \fpar{r_1}{g_{1j}}\biggr)\\
&= \frac{\beta}{\sqrt{N}} \ee\bigl(\avg{\xxi_j\sigma_1\sigma_j} f(r_1)\bigr) + \frac{1}{\sqrt{N}}\ee\bigl(\avg{\sigma_j}\avg{\eta_1 \sigma_j}f'(r_1)\bigr).
\end{align*}
Thus, we have
\begin{equation}\label{step2}
\begin{split}
\ee\biggl(\frac{1}{\sqrt{N}}\sum_{j=2}^N g_{1j}\avg{\sigma_j} f(r_1)\biggr) &= \ee\biggl(\beta  f(r_1)\frac{\sum_{j=2}^N \avg{\xxi_j \sigma_1\sigma_j}}{N}\biggr)\\
&\quad  + \ee\biggl(f'(r_1) \frac{\sum_{j=2}^N \avg{\sigma_j}\avg{\eta_1\sigma_j}}{N}\biggr).
\end{split}
\end{equation}
By \eqref{est2} and the bound on $|f|$, we have
\begin{equation}\label{step3}
\biggl|\ee\biggl(\beta  f(r_1)\frac{\sum_{j=2}^N \avg{\xxi_j \sigma_1\sigma_j}}{N}\biggr) - \beta(1-q) \ee\bigl(f(r_1)\avg{\sigma_1}\bigr)\biggr| \le \frac{C(\beta,h)}{\sqrt{N}}.
\end{equation}
Similarly, by \eqref{est30} and the bound on $|f'|$, 
\begin{equation}\label{step4}
\biggl|\ee\biggl(f'(r_1) \frac{\sum_{j=2}^N \avg{\sigma_j}\avg{\eta_1\sigma_j}}{N}\biggr) - q\ee\bigl(f'(r_1)\avg{\eta_1}\bigr)\biggr| \le \frac{C(\beta,h)}{\sqrt{N}}.
\end{equation}
Combining \eqref{step1}, \eqref{step2}, \eqref{step3}, and \eqref{step4}, we get
\begin{equation*}\label{step5}
\bigl|\ee(r_1 f(r_1)) - q\ee\bigl(f'(r_1)\avg{\eta_1}\bigr)| \le \frac{C(\beta,h)}{\sqrt{N}}.
\end{equation*}
Since $f'$ is bounded and
\begin{align*}
\ee u(r_1) - \ee u(z\sqrt{q}) &= q^{-1}\ee(qf'(r_1)-r_1 f(r_1)),
\end{align*}
the proof will be complete if we can show that 
\begin{equation}\label{finalstep}
\ee\bigl(\avg{\eta_1} - 1\bigr)^2 \le \frac{C(\beta,h)}{\sqrt{N}}.
\end{equation}
The rest of the proof is devoted to proving \eqref{finalstep}. From the definition \eqref{eta1def} of $\eta_1$, we get
\[
\avg{\eta_1} = 1 + \frac{\beta}{\sqrt{N}}\sum_{j=2}^N g_{1j} \avg{\xxi_j \sigma_1} - \beta^2 (1-q) \avg{\xxi_1\sigma_1}.
\]
Now let 
\[
h_j = \frac{\beta\avg{\xxi_j \sigma_1}}{\sqrt{N}},
\]
so that
\[
\avg{\eta_1} = 1 + \sum_{j=2}^N g_{1j} h_j - \beta^2(1-q)\bigl(1-\avg{\sigma_1}^2\bigr).
\]
Our intention is to apply Lemma \ref{mainlmm} to show that the second and the third terms approximately cancel each other out. First, note that by equations ~\eqref{deriv1} and \eqref{xider}, 
\begin{equation}\label{fjder}
\begin{split}
\fpar{h_j}{g_{1k}} = \frac{\beta^2}{N} \bigl(-\avg{\xxi_j \sigma_1 \sigma_k} \avg{\sigma_1} + \avg{\xxi_j \sigma_k} - \avg{\xxi_j\sigma_1}\avg{\sigma_1\sigma_k}\bigr).
\end{split}
\end{equation}
Thus, if we put
\begin{align*}
v_k &= -\sigma_1 \sigma_k \avg{\sigma_1} + \sigma_k - \sigma_1 \avg{\sigma_1 \sigma_k}, \ \ \text{and}\\
w_j &= -\xxi_j \sigma_1 \avg{\sigma_1} + \xxi_j - \avg{\xxi_j \sigma_1}\sigma_1,
\end{align*}
then
\[
\fpar{h_j}{g_{1k}} = \frac{\beta^2\avg{\xxi_j v_k}}{N} = \frac{\beta^2\avg{w_j \sigma_k}}{N}.
\]
From Lemma \ref{jkbd} it follows that
\[
\sum_{j,k=2}^N \ee\biggl(\fpar{h_j}{g_{1k}}\fpar{h_k}{g_{1j}}\biggr) \le \frac{C(\beta,h)}{\sqrt{N}}.
\]
Again, by \eqref{est1} we get
\[
\sum_{j=2}^N \ee h_j^2 = \frac{\beta^2}{N}\sum_{j=2}^N \ee \avg{\xxi_j \sigma_1}^2 \le \frac{C(\beta,h)}{\sqrt{N}}.
\]
Using the last two bounds in Lemma \ref{mainlmm}, we have
\begin{equation}\label{app1}
\ee\biggl(\sum_{j=2}^N g_{1j} h_j - \sum_{j=2}^N \fpar{h_j}{g_{1j}}\biggr)^2 \le \frac{C(\beta,h)}{\sqrt{N}}.
\end{equation}
Now from \eqref{fjder} we see that
\begin{align*}
\sum_{j=2}^N \fpar{h_j}{g_{1j}} &=  -\frac{\beta^2\avg{\sigma_1}}{N} \sum_{j=2}^N\avg{\xxi_j \sigma_1 \sigma_j} + \frac{\beta^2}{N}\sum_{j=2}^N\avg{\xxi_j \sigma_j} \\
&\quad \quad - \frac{\beta^2}{N}\sum_{j=2}^N\avg{\xxi_j\sigma_1}\avg{\sigma_1\sigma_j}.
\end{align*}
Applying \eqref{est2}  to the first two terms and \eqref{est3} to the third term, we get
\begin{align}\label{app2}
\ee\biggl(\sum_{j=2}^N \fpar{h_j}{g_{1j}} + \beta^2(1-q)\avg{\sigma_1}^2 - \beta^2(1-q)\biggr)^2 &\le \frac{C(\beta,h)}{N}.
\end{align}
Combining \eqref{app1} and \eqref{app2}, we have
\begin{align*}
\ee\bigl(\avg{\eta_1} - 1\bigr)^2 &= \ee\biggl(\sum_{j=2}^N g_{1j} h_j - \beta^2(1-q)\bigl(1-\avg{\sigma_1}^2\bigr)\biggr)^2  \le \frac{C(\beta,h)}{\sqrt{N}}.
\end{align*}
This proves \eqref{finalstep} and hence completes the proof of Theorem \ref{newthm}.

\subsection{Proof of Lemma \ref{steinlmm}}
For each $x,\mu\in \rr$, let
\[
\rho(x,\mu) = \cosh(ax + b) e^{-(x-\mu)^2/2\sigma^2}
\]
and 
\[
r(\mu) = \int_{-\infty}^\infty u(x)\psi_{a,b,\mu,\sigma^2}(x) dx = \frac{\int_{-\infty}^\infty u(x)\rho(x,\mu) dx}{\int_{-\infty}^\infty \rho(x, \mu) dx}.
\]
Let
\begin{align*}
f(x,\mu) &= \frac{1}{\rho(x,\mu)} \int_{-\infty}^x \rho(t,\mu)(u(t)-r(\mu)) dt \\
&= -\frac{1}{\rho(x,\mu)} \int_x^\infty \rho(t,\mu)(u(t)-r(\mu)) dt.
\end{align*}
For ease of notation, let us write $\fpar{f}{x}$ and $\fpar{f}{\mu}$ as $f_1$ and $f_2$.
Multiplying by $\rho(x,\mu)$ on both sides and differentiating w.r.t.\ $x$, and finally dividing everything by $\rho(x,\mu)$, we get
\begin{equation}\label{lasteq}
f_1(x,\mu) - \biggl(\frac{x-\mu}{\sigma^2} - a\tanh(ax+b)\biggr) f(x, \mu) = u(x) - r(\mu).
\end{equation}
This proves the first assertion of the lemma. Now  observe that for any $x \ge \mu$,
\begin{align}\label{firstbd}
&\int_x^\infty \exp\biggl(at+b - \frac{(t-\mu)^2}{2\sigma^2}\biggr) dt \nonumber \\
&= e^{b + a\mu + \frac{1}{2}a^2\sigma^2}\int_x^\infty \exp\biggl( - \frac{(t-\mu- a\sigma^2)^2}{2\sigma^2}\biggr) dt\nonumber\\
&=e^{b + a\mu + \frac{1}{2}a^2\sigma^2}\int_{\frac{x-\mu-a\sigma^2}{\sigma}}^\infty e^{-\frac{1}{2}y^2}  \sigma dy.
\end{align}
Let 
\[
C_1(a,\sigma) = \sigma \sup_{z\ge -a\sigma} e^{\frac{1}{2}z^2}\int_z^\infty e^{-\frac{1}{2}y^2} dy.
\]
It is easy to verify by elementary calculus that $C_1(a,\sigma)$ is finite. Since $x\ge \mu$, we now get from \eqref{firstbd} that 
\begin{align*}
&\int_x^\infty \exp\biggl(at+b - \frac{(t-\mu)^2}{2\sigma^2}\biggr) dt\\
&\le C_1(a,\sigma) e^{b + a\mu + \frac{1}{2}a^2\sigma^2}\exp\biggl( - \frac{(x-\mu- a\sigma^2)^2}{2\sigma^2}\biggr) \\
&= C_1(a,\sigma) \exp\biggl(ax+b - \frac{(x-\mu)^2}{2\sigma^2}\biggr).
\end{align*}
Similarly, we have
\[
\int_x^\infty \exp\biggl(-at-b - \frac{(t-\mu)^2}{2\sigma^2}\biggr) dt \le C_1(-a,\sigma) \exp\biggl(-ax-b - \frac{(x-\mu)^2}{2\sigma^2}\biggr).
\]
Combining, and putting $C_2(a,\sigma) = \max\{C_1(a,\sigma), C_1(-a,\sigma)\}$, we get that for $x\ge \mu$,
\begin{equation}\label{rhotail}
\int_x^\infty \rho(t,\mu) dt \le C_2(a,\sigma) \rho(x,\mu).
\end{equation}
Similarly, if $x < \mu$, we have
\[
\int_{-\infty}^x \rho(t,\mu) dt \le C_3(a,\sigma) \rho(x,\mu)
\]
for some other constant $C_3(a,\sigma)$. From the definition of $f$, we can now deduce that
\[
\|f\|_\infty \le C_4(a,\sigma) \|u\|_\infty,
\]
where $C_4 = 2\max\{C_2, C_3\}$.
Next, note that for $x \ge \mu$, 
\begin{align*}
&\int_x^\infty |t-\mu|\exp\biggl(at+b - \frac{(t-\mu)^2}{2\sigma^2}\biggr) dt\\
&= e^{b + a\mu + \frac{1}{2}a^2\sigma^2}\int_x^\infty (t-\mu)\exp\biggl( - \frac{(t-\mu- a\sigma^2)^2}{2\sigma^2}\biggr) dt\\
&=e^{b + a\mu + \frac{1}{2}a^2\sigma^2}\int_{\frac{x-\mu-a\sigma^2}{\sigma}}^\infty (\sigma y+a\sigma^2)e^{-\frac{1}{2}y^2}  \sigma dy.
\end{align*}
Putting
\[
C_5(a,\sigma) = \sigma \sup_{z\ge -a\sigma} e^{\frac{1}{2}z^2}\int_z^\infty(\sigma y+a\sigma^2) e^{-\frac{1}{2}y^2} dy = \sigma^2 + a\sigma^2 C_1(a,\sigma),
\]
we get 
\begin{align*}
&\int_x^\infty |t-\mu|\exp\biggl(at+b - \frac{(t-\mu)^2}{2\sigma^2}\biggr) dt\\
 &\le C_5(a,\sigma) \exp\biggl(ax+b - \frac{(x-\mu)^2}{2\sigma^2}\biggr).
\end{align*}
Proceeding as before, this leads to
\begin{equation}\label{secondbd}
\int_x^\infty|t-\mu| \rho(t,\mu) dt \le C_6(a,\sigma) \rho(x,\mu),
\end{equation}
and a similar bound for the integral from $-\infty$ to $x$ in the case $x<\mu$. An easy byproduct of these inequalities is the bound
\begin{equation}\label{thirdbd}
\sup_{x\in \rr}|(x-\mu) f(x,\mu)| \le C_7(a,\sigma)\|u\|_\infty.
\end{equation}
Using this and \eqref{lasteq}, and the previous deduction that $\|f\|_\infty \le C_4(a,\sigma)\|u\|_\infty$, it follows that
\[
|f_1(x,\mu)|\le C_8(a,\sigma)\|u\|_\infty.
\]
Now note that
\[
\frac{1}{\rho}\fpar{\rho}{\mu}= \frac{x-\mu}{\sigma^2}.
\]
Thus, for $x\ge \mu$, we have
\begin{align*}
f_2(x,\mu) &= -\frac{x-\mu}{\sigma^2}f(x,\mu) - \frac{1}{\rho(x,\mu)}\int_x^\infty \frac{t-\mu}{\sigma^2}\rho(t,\mu)(u(t)-r(\mu)) dt \\
&\qquad + \frac{r'(\mu)}{\rho(x,\mu)}\int_x^\infty \rho(t,\mu) dt.
\end{align*}
Thus from \eqref{thirdbd}, \eqref{secondbd} and \eqref{rhotail} it follows that
\[
|f_2(x,\mu)|\le C_9(a,\sigma) \|u\|_\infty + C_2(a,\sigma)|r'(\mu)|.
\]
A simple computation gives
\begin{align*}
r'(\mu) = \frac{\int_{-\infty}^\infty (u(t)-r(\mu)) \rho(t,\mu) \frac{t-\mu}{\sigma^2} dt}{\int_{-\infty}^\infty \rho(t,\mu) dt}.
\end{align*}
Thus,
\[
|r'(\mu)| \le \frac{2\|u\|_\infty}{\sigma^2}\int_{-\infty}^\infty |t-\mu|\psi_{a,b,\mu,\sigma^2}(t) dt.
\]
From the representation \eqref{mix} 
it follows that 
\[
\int_{-\infty}^\infty |t-\mu|\psi_{a,b,\mu,\sigma^2}(t) dt \le C_{10}(a,\sigma).
\]
Thus, for $x\ge \mu$, $|f_2(x,\mu)| \le C_{11}(a,\sigma) \|u\|_\infty$. 
The bound for $x<\mu$ follows similarly.

\vskip.5in
\noindent{\bf Acknowledgments.} The author thanks Michel Talagrand, Persi Diaconis  and the associate editor for various helpful suggestions. The author is also grateful to the referee for a very careful reading of the proofs and a large number of useful comments.

\end{document}